\documentclass[american]{scrartcl}
\KOMAoptions{DIV=14, paper=a4, abstract=on}
\usepackage[utf8]{inputenc}
\usepackage[T1]{fontenc}
\usepackage{babel, csquotes, lmodern, microtype, fixltx2e}
\usepackage{amsmath, amssymb, amsthm, subcaption}
\usepackage{pgfplots}
\usepackage{hyperref}
\hypersetup{%
   pdfauthor  ={Lennard Kamenski, Weizhang Huang},
   pdftitle   ={How a nonconvergent recovered Hessian works
      in mesh adaptation},
   pdfsubject ={AMS MSC: 65N50, 65N30}
}

\usepackage[capitalize]{cleveref}
\crefformat{section}{#2Sect.~#1#3}
\crefmultiformat{section}{Sects.~#2#1#3}{ and~#2#1#3}{, #2#1#3}{ and~#2#1#3}
\crefformat{equation}{(#2#1#3)}
\crefmultiformat{equation}{(#2#1#3)}{ and~(#2#1#3)}{, (#2#1#3)}{ and~(#2#1#3)}
\crefrangeformat{equation}{(#3#1#4) to~(#5#2#6)}

\makeatletter\g@addto@macro\@floatboxreset\centering\makeatother

\newcommand{\bx}{\boldsymbol{x}}

\DeclareMathOperator{\tr}{tr}
\newcommand{\abs}[1]{\lvert#1\rvert}
\newcommand{\Abs}[1]{\left\lvert#1\right\rvert}
\newcommand{\norm}[1]{\lVert#1\rVert}
\newcommand{\Norm}[1]{\left\lVert#1\right\rVert}
\newcommand{\dx}{\,d\bx}

\newcommand{\FHF}{{(F_K')}^T \Abs{H (\bx)} F_K'}
\newcommand{\FMF}{{(F_K')}^T M_K F_K'}
\newcommand{\FRF}{{(F_K')}^T R_K F_K'}

\newcommand{\FRFinv}{{(F_K')}^{-T} R_K^{-1} {(F_K')}^{-1}}
\newcommand{\RKinvH}{R_K^{-1} H (\bx)}
\newcommand{\tK}{{\tilde{K}}}
\newcommand{\cO}{\mathcal{O}}
\newcommand{\R}{\mathbb{R}}
\newcommand{\cT}{\mathcal{T}}

\newtheorem{theorem}{Theorem}[section]
\newtheorem{example}{Example}[section]

\providecommand{\bamg}{\texttt{BAMG}}

\newenvironment{keywords}%
   {\begin{trivlist}\item[]{\bfseries\sffamily Key words:}~}
   {\end{trivlist}}
\newenvironment{AMS}%
   {\begin{trivlist}\item[]{\bfseries\sffamily AMS subject classifications:}~}
   {\end{trivlist}}

\begin{document}

\title{How a~nonconvergent recovered Hessian works in~mesh adaptation%
   \thanks{%
      Supported in part by 
      the DFG under Grant KA\,3215/1-2
      and
      the NSF under Grant DMS-1115118.%
      }%
}

\author{Lennard Kamenski%
   \thanks{%
      Weierstrass Institute, Berlin, Germany
      (\href{mailto:kamenski@wias-berlin.de}{\nolinkurl{kamenski@wias-berlin.de}}).%
   }
   \and
   Weizhang Huang%
   \thanks{%
      University of Kansas, Department of Mathematics, Lawrence, KS 66045, USA
      (\href{mailto:whuang@ku.edu}{\nolinkurl{whuang@ku.edu}}).%
   }
}

\begin{titlepage}
\maketitle
\begin{abstract}
Hessian recovery has been commonly used in mesh adaptation for obtaining the required magnitude and direction information of the solution error.
Unfortunately, a recovered Hessian from a linear finite element approximation is nonconvergent in general as the mesh is refined.
It has been observed numerically that adaptive meshes based on such a nonconvergent recovered Hessian can nevertheless lead to an optimal error in the finite element approximation.
This also explains why Hessian recovery is still widely used despite its nonconvergence.
In this paper we develop an error bound for the linear finite element solution of a general boundary value problem under a mild assumption on the closeness of the recovered Hessian to the exact one.
Numerical results show that this closeness assumption is satisfied by the recovered Hessian obtained with commonly used Hessian recovery methods.
Moreover, it is shown that the finite element error changes gradually with the closeness of the recovered Hessian.
This provides an explanation on how a nonconvergent recovered Hessian works in mesh adaptation.


\end{abstract}

\end{titlepage}

\section{Introduction}
\label{sect:introduction}

Gradient and Hessian recovery has been commonly used in mesh adaptation for the numerical solution of partial differential equations (PDEs); e.g.,\ see~\cite{AinOde00,BabStr01,HuaRus11,Tang07,ZhaNag05,ZieZhu92,ZieZhu92a}.
The use typically involves the approximation of solution derivatives based on a computed solution defined on the current mesh (recovery), the generation of a new mesh using the recovered derivatives, and the solution of the physical PDE on the new mesh.
These steps are often repeated several times until a suitable mesh and a numerical solution defined thereon are obtained.
As the mesh is refined, a sequence of adaptive meshes, derivative approximations, and numerical solutions results.
A theoretical and also practical question is whether this sequence of numerical solutions converges to the exact solution.
Naturally, this question is linked to the convergence of the recovered derivatives used to generate the meshes. 
It is known that recovered gradient through the least squares fitting~\cite{ZieZhu92,ZieZhu92a} or polynomial preserving techniques~\cite{ZhaNag05} is convergent for uniform or quasi-uniform meshes~\cite{ZhaNag05,ZhaZhu95} and superconvergent for mildly structured meshes~\cite{XuZha04} as well for a type of adaptive mesh~\cite{WuZha07}.

For the Hessian, it has been observed that, unfortunately, a convergent recovery cannot be obtained from linear finite element approximations for general nonuniform meshes~\cite{AgoLipVas10,Kam09,PicAlaBorGeo11}, although Hessian recovery is known to converge when the numerical solution exhibits superconvergence or supercloseness for some special meshes~\cite{BanXu03,BanXu03a,Ova07}.

On the other hand, numerical experiments also show that the numerical solution obtained with an adaptive mesh generated using a nonconvergent recovered Hessian is often not only convergent but also has an error comparable to that obtained with the exact analytical Hessian.
To demonstrate this, we consider a Dirichlet boundary value problem (BVP) for the Poisson equation
\begin{equation}
   \begin{cases} 
      -\Delta u = f,  &\text{in $\Omega= (0,1)\times(0,1)$}, \\
              u = g,  &\text{on $\partial\Omega$},
   \end{cases} 
   \label{eq:bvp}
\end{equation}
where $f$ and $g$ are chosen such that the exact solution of the BVP is given by
\begin{equation}
   u(x,y) = x^2 + 25 y^2.
\end{equation}
Two Hessian recovery methods, QLS (quadratic least squares fitting) and WF (weak formulation) are used (see~\cref{sect:recovery:methods} for the description of these and other Hessian recovery techniques).
\Cref{fig:x2y2:25:intro} shows the error in recovered Hessian and the linear finite element solution with exact and recovered Hessian.
One can see that the finite element error is convergent and almost undistinguishable for the exact and approximate Hessian (\cref{fig:x2y2:25:solution}) whereas the error of the Hessian recovery remains $\cO(1)$ (\cref{fig:x2y2:25:recovery}).
Obviously, this indicates that a convergent recovered Hessian is not necessary for the purpose of mesh adaptation.
Of course, a badly recovered Hessian does not serve the purpose either.

\begin{figure}[t]
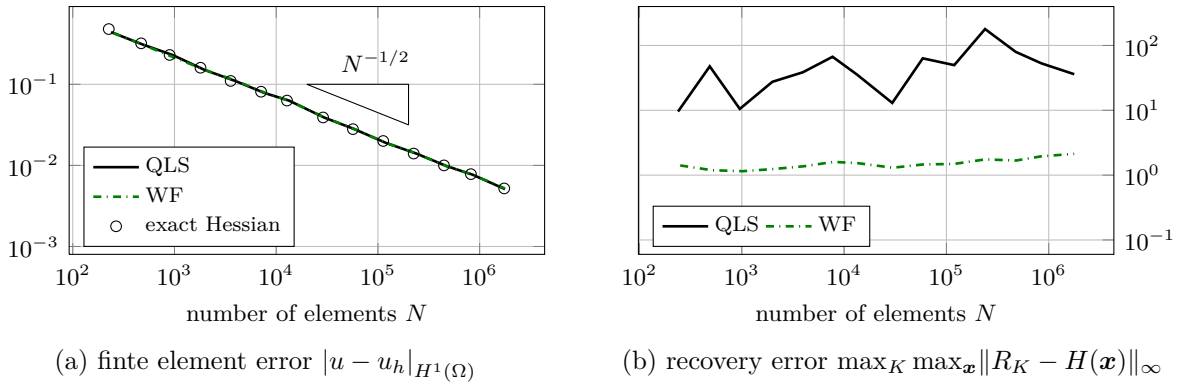

   \begin{subfigure}[t]{0.5\textwidth}
      \centering
      \includegraphics[clip]{{{1211.2877f-intro-fe-error}}}
      \caption{finte element error $\Abs{u-u_h}_{H^1(\Omega)}$\label{fig:x2y2:25:solution}}
   \end{subfigure}%
   \begin{subfigure}[t]{0.5\textwidth}
      \centering
      \includegraphics[clip]{{{1211.2877f-intro-h-error}}}
      \caption{recovery error $\max_K \max_{\bx} \norm{R_{K} - H(\bx)}_\infty$\label{fig:x2y2:25:recovery}}
   \end{subfigure}
   \caption{Finite element and Hessian recovery errors
      as a function of $N$\label{fig:x2y2:25:intro}}
\end{figure}

How accurate should a recovered Hessian be for the purpose of mesh adaptation? 
This issue has been studied by Agouzal et al.~\cite{AgoLipVas99} and Vassilevski and Lipnikov~\cite{VasLip99}.
In particular~\cite[Theorem~3.2]{AgoLipVas99}, they show that a mesh based on an approximation $R$ of the Hessian $H$ is quasi-optimal if there exist small (with respect to one) positive numbers $\varepsilon$ and $\delta$ such that
\begin{align}
  \max_{\bx \in \omega_i} \Norm{ H(\bx) - H_{\omega_i} }_\infty & \leq \delta \lambda_{\min} \bigl(R(\bx_i)\bigr),
   \label{eq:AgLiVa99:1} \\
   \Norm{ R(\bx_i) - H_{\omega_i} }_\infty &\leq \varepsilon \lambda_{\min} \bigl(R(\bx_i)\bigr) 
   \label{eq:AgLiVa99:2}
\end{align}
hold for any mesh vertex $\bx_i$ and its patch  $\omega_i$, where $H_{\omega_i}$ is the Hessian at a point in $\omega_i$ where $\Abs{\det H(\bx) }$ attains its maximum and $\lambda_{\min}(\cdot)$ denotes the minimum eigenvalue of a matrix.
Notice that \cref{eq:AgLiVa99:2} does not require $R$ to converge to $H$ as the mesh is refined.
Instead, it requires the eigenvalues of $R^{-1} H$ to be around one (cf.~\cref{sect:analysis:1}).
Unfortunately, it is still too restrictive to be satisfied by most examples we tested; see~\cref{sect:examples}.
Thus, the work~\cite{AgoLipVas99,VasLip99} does not give a full explanation why a nonconvergent recovered Hessian works in mesh adaptation.

The objective of the paper is to present a study on this issue.
To be specific, we consider a BVP and its linear finite element solution with adaptive anisotropic meshes generated from a recovered Hessian.
We adopt the $M$-uniform mesh approach~\cite{Hua07,HuaRus11} to view any adaptive mesh as a uniform one in some metric depending on the computed solution.
An advantage of the approach is that the relation between the recovered Hessian and an adaptive anisotropic mesh generated using it can be fully characterized through the so-called alignment and equidistribution conditions (see \cref{eq:equi,eq:ali} in~\cref{sect:analysis:1}).
This characterization plays a crucial role in the development of a bound for the $H^1$ semi-norm of the finite element error.
The bound converges at a first order rate in terms of the average element diameter, $N^{-\frac{1}{d}}$, where $N$ is the number of elements and $d$ is the dimension of the physical domain. 
Moreover, the bound is valid under a condition on the closeness of the recovered Hessian to the exact one; see \cref{eq:CRs} or \cref{eq:CRplus:general,eq:CRminus:general}.
This closeness condition is much weaker than \cref{eq:AgLiVa99:2}.
Roughly speaking, \cref{eq:AgLiVa99:2} requires the eigenvalues of $R^{-1} H$ to be around one whereas the new condition only requires them to be bounded below from zero and from above.
Numerical results in~\cref{sect:examples} show that the new closeness condition is satisfied in all examples for four commonly used Hessian recovery techniques considered in this paper whereas \cref{eq:AgLiVa99:2} is satisfied only in some examples.
Furthermore, the error bound is linearly proportional to the ratio of the maximum (over the physical domain) of the largest eigenvalues of $R^{-1} H$ to the minimum of the smallest eigenvalues.
Since the ratio is a measure of the closeness of the recovered Hessian to the exact one, the dependence indicates that the finite element error changes gradually with the closeness of the recovered Hessian.
Hence, the error for the linear finite element approximation of the BVP is convergent for the considered Hessian recovery techniques and insensitive to the closeness of the recovered Hessian to the exact one.
This provides an explanation how a nonconvergent recovered Hessian works for mesh adaptation.

An outline of the paper is as follows. Convergence analysis of the linear finite element approximation is given in~\cref{sect:analysis:1,sect:general} for the cases with positive definite and general Hessian, respectively.
A brief description of four common Hessian recovery techniques is given in~\cref{sect:recovery:methods} followed by numerical examples in~\cref{sect:examples}.
Finally, \cref{sect:conclusion} contains conclusions and further comments.

\section{Convergence of~linear finite element approximation for~positive definite Hessian}
\label{sect:analysis:1}

We consider the BVP
\begin{equation}
   \begin{cases}
   \mathcal{L} u = f,  & \text{ in $\Omega$}, \\
               u = g,  & \text{ on $\partial\Omega$},
   \end{cases}
   \label{eq:bvp-2}
\end{equation}
where $\Omega$ is a polygonal or polyhedral domain of $\R^d$ ($d \ge 1$), $\mathcal{L}$ is an elliptic second-order differential operator, and $f$ and $g$ are given functions.
We are concerned with the adaptive mesh solution of this BVP using the conventional linear finite element method.
Denote a family of simplicial meshes for $\Omega$ by $\{ \cT_h\}$ and the corresponding reference element by $\hat{K}$ which is chosen to be unitary in volume.
For each mesh $\cT_h$, we denote the corresponding finite element solution by $u_h$.
C\'{e}a's lemma implies that the finite element error is bounded by the interpolation error, i.e.,
\begin{equation}
   \Abs{u-u_h}_{H^1(\Omega)} \le C \Abs{u - \Pi_h u}_{H^1(\Omega)} ,
   \label{eq:cea-1}
\end{equation}
where $C$ is a constant independent of $u$ and $\mathcal{T}_h$ and $\Pi_h$ is the nodal interpolation operator associated with the linear finite element space defined on $\mathcal{T}_h$.
Note that \cref{eq:cea-1} is valid for any mesh.

\subsection{\texorpdfstring{Quasi-$M$-uniform meshes}{Quasi-M-uniform meshes}}
In this paper we consider adaptive meshes generated based on a recovered Hessian $R$ and use the $M$-uniform mesh approach with which any adaptive mesh is viewed as a uniform one in some metric $M$ (defined in terms of $R$ in our current situation).
It is known~\cite{Hua07,HuaRus11} that such an $M$-uniform mesh satisfies the equidistribution and alignment conditions,
\begin{align}
   \Abs{K} {\det(M_K)}^{\frac{1}{2}} 
      &= \frac{1}{N} \sum_{\tK\in\cT_h} \Abs{\tK} {\det(M_{\tK})}^{\frac{1}{2}},
         \quad \forall K \in \cT_h,
   \label{eq:equi} \\
   \frac{1}{d} \tr\left( \FMF \right) 
      &= {\det\left( \FMF \right)}^{\frac{1}{d}},
      \quad \forall K \in \cT_h,
   \label{eq:ali}
\end{align}
where $N$ is the number of mesh elements, $M_K$ is an average of $M$ over $K$, $F_K \colon \hat{K} \to K$ is the affine mapping from the reference element $\hat{K}$ to a mesh element $K$, $F_K'$ is the Jacobian matrix of $F_K$ (which is constant on $K$), and $\det(\cdot)$ and $\tr(\cdot)$ denote the determinant and trace of a matrix, respectively.

In practice, it is more realistic to generate less restrictive quasi-$M$-uniform meshes which satisfy
\begin{align}
   \Abs{K} {\det(M_K)}^{\frac{1}{2}}
      & \leq  C_{eq} \frac{1}{N} \sum_{\tK\in\cT_h} 
      \Abs{\tK} {\det(M_{\tK})}^{\frac{1}{2}},
   \quad \forall K \in \cT_h, 
   \label{eq:equi:approx}
   \\
   \frac{1}{d} \tr\left( \FMF \right) 
      & \leq C_{ali} \Abs{K}^{\frac{2}{d}} {\det(M_K)}^{\frac{1}{d}},
   \quad \forall K \in \cT_h,
   \label{eq:ali:approx}
\end{align}
where $C_{eq}, C_{ali} \geq 1$ are some constants independent of $K$, $N$, and  $\cT_h$.
Numerical experiments in~\cite{Hua05a} and~\cref{sect:examples} (\cref{fig:flower:ceq,fig:flower:cali,fig:tanh:ceq,fig:tanh:cali}) show that quasi-$M$-uniform meshes with relatively small $C_{eq}$ and $C_{ali}$ can be generated in practice.
For this reason, we use quasi-$M$-uniform meshes in our analysis and numerical experiments.

We would like to point out that conditions \cref{eq:equi:approx,eq:ali:approx} with $C_{eq} = C_{ali} = 1$ imply \cref{eq:equi,eq:ali}.
Indeed, the inequality \cref{eq:ali:approx} with $C_{ali} = 1$ becomes the equality \cref{eq:ali} because the left-hand side of it (the arithmetic mean of the eigenvalues of $\FMF$) cannot be smaller than the right-hand side (the geometric mean of the eigenvalues).
Further, if $C_{eq} = 1$ then \cref{eq:equi:approx} becomes
\[
   \Abs{K} {\det(M_K)}^{\frac{1}{2}} 
      \le \frac{1}{N} \sum_{\tK\in\cT_h} 
         \Abs{\tK} {\det(M_{\tK})}^{\frac{1}{2}},
   \quad \forall K \in \cT_h.
\]
This implies
\begin{align*}
 \max_{K\in\cT_h} \Abs{K} {\det(M_K)}^{\frac{1}{2}} 
      &\leq \frac{1}{N} \sum_{K\in\cT_h} \Abs{K} {\det(M_{K})}^{\frac{1}{2}}\\
      &\leq \frac{1}{N} \left( (N-1)\max_{K\in\cT_h} \Abs{K} {\det(M_K)}^{\frac{1}{2}}
         + \min_{K\in\cT_h} \Abs{K} {\det(M_K)}^{\frac{1}{2}}
   \right)
\end{align*}
and therefore
\[
   \max_{K\in\cT_h} \Abs{K} {\det(M_K)}^{\frac{1}{2}}
      \leq \min_{K\in\cT_h} \Abs{K} {\det(M_K)}^{\frac{1}{2}},
\]
which can only be valid if all values of $\Abs{K} {\det(M_K)}^{\frac{1}{2}}$ are the same for all $K$.

\subsection{Main result}
In this section we consider a special case where the Hessian of the solution is uniformly positive definite in $\Omega$; i.e.,
\begin{equation}
   \exists \gamma > 0 \colon H(\bx) \geq \gamma I, \quad \forall \bx \in \Omega,
   \label{eq:Hpd}
\end{equation}
where the greater-than-or-equal sign means that the difference between the left-hand side and right-hand side terms is positive semidefinite. 
We also assume that the recovered Hessian $R$ is uniformly positive definite in $\Omega$.
This assumption is not essential and will be dropped for the general situation discussed in~\cref{sect:general}.

Recall from \cref{eq:cea-1} that the finite element error is bounded by the $H^1$ semi-norm of the interpolation error of the exact solution.
A metric tensor corresponding to the $H^1$ semi-norm can be defined as
\begin{equation}
   M_K = {\det(R_K)}^{- \frac{1}{d+2}} \Norm{R_K}_2^{\frac{2}{d+2}} R_K,
   \quad \forall K \in \cT_h,
\label{eq:M:H1}
\end{equation}
where $R_K$ is an average of $R$ over $K$~\cite{Hua05a}.
For this metric tensor,  mesh conditions \cref{eq:equi:approx,eq:ali:approx} become
\begin{align}
   \Abs{K} {\det(R_K)}^{\frac{1}{d+2}} \Norm{R_K}_2^{\frac{d}{d+2}}
   &\leq C_{eq} \frac{1}{N}
      \sum_\tK \abs{\tK} {\det(R_\tK)}^{\frac{1}{d+2}} \Norm{R_\tK}_2^{\frac{d}{d+2}},
   \quad \forall K \in \cT_h,
   \label{eq:equi:H1}
   \\
   \frac{1}{d} \tr\left( \FRF \right)
   &\leq C_{ali} \Abs{K}^{\frac{2}{d}} {\det(R_K)}^{\frac{1}{d}},
   \quad \forall K \in \cT_h.
   \label{eq:ali:H1}
\end{align}

Note that the alignment condition \cref{eq:ali:H1} implies the inverse alignment condition
\begin{equation}
   \frac{1}{d} \tr\left( \FRFinv \right)
      <
         {\left( \frac{d}{d-1} C_{ali}\right)}^{d-1}
         \Abs{K}^{-\frac{2}{d}} {\det(R_K)}^{-\frac{1}{d}},
   \quad \forall K \in \cT_h .
   \label{eq:ali:inverse}
\end{equation}
To show this, we denote the eigenvalues of $\FRF$ by $0 < \lambda_1 \le \cdots \le \lambda_d$ and rewrite \cref{eq:ali:H1} as
\[
   \sum_i \lambda_i 
      \le d C_{ali} {\left(\prod_i \lambda_i\right)}^{\frac{1}{d}}.
\]
Then \cref{eq:ali:inverse} follows from
\begin{align*}
   \frac{1}{d} \sum_i \lambda_i^{-1} 
   &= \prod_i \lambda_i^{-1} \cdot \frac{1}{d} 
      \sum_i \prod_{j\neq i} \lambda_j \\
   &\leq \prod_i \lambda_i^{-1}
      \cdot \frac{1}{d} \sum_i 
         {\left(\frac{\sum_{j\neq i} \lambda_j}{d-1}  \right)}^{d-1}\\
   &< \prod_i \lambda_i^{-1} 
      \cdot \frac{1}{d} \sum_i 
         {\left( \frac{ \sum_{j} \lambda_j }{d-1}\right)}^{d-1}
   = \prod_i \lambda_i^{-1} 
         {\left( \frac{\sum_{j} \lambda_j }{d-1} \right)}^{d-1}\\
   &\leq {\left( \frac{d}{d-1} C_{ali} \right)}^{d-1}
      {\left(\prod_i \lambda_i\right)}^{-\frac{1}{d}}
   .
\end{align*}

\begin{theorem}[Positive definite Hessian]
\label{thm:H1}
Assume that $H(\bx)$ and the recovered Hessian $R$ are uniformly positive definite in $\Omega$ and that $R$ satisfies
\begin{equation}
   C_{R-,K} I \leq \RKinvH \leq C_{R+, K} I,
   \quad \forall \bx \in K,
   \quad \forall K \in \cT_h
   \label{eq:CRs}
\end{equation}
where $C_{R-,K}$ and $C_{R+,K}$ are element-wise constants satisfying
\begin{equation}
   C_{R-} \le \min_{K \in \cT_h} C_{R-,K}
   \qquad \text{and} \qquad
\sqrt{\frac{1}{N} \sum_{K\in \cT_h} C_{R+,K}^2} \le C_{R+}
\label{CR+}
\end{equation}
with some mesh-independent positive constants $C_{R-}$ and $C_{R+}$.
If the solution of the BVP \cref{eq:bvp-2} is in $H^2(\Omega)$,
then for any quasi-$M$-uniform mesh associated with the metric tensor \cref{eq:M:H1} and satisfying \cref{eq:equi:approx,eq:ali:approx} the linear finite element error for the BVP is bounded by
\begin{equation}
   \Abs{u-u_h}_{H^1(\Omega)} 
   \leq  C 
      \cdot C_{ali}^{\frac{d+1}{2}} C_{eq}^{\frac{d+2}{2d}} 
      \cdot \frac{C_{R+}}{C_{R-}}
      \cdot N^{-\frac{1}{d}}
      \Norm{ {\det(H)}^{\frac{1}{d}} H }_{L^{\frac{d}{d+2}}(\Omega)}^{\frac{1}{2}}.
   \label{eq:thm:H1}
\end{equation}
\end{theorem}

\begin{proof}
The nodal interpolation error of a function $u \in H^2(\Omega)$ on $K$ is bounded by
\begin{equation}
   \Abs{u - \Pi_h u}_{H^1(K)} 
      \leq C \Norm{ {(F_K')}^{-1}}_2 
         {\left( \int_K \Norm{\FHF}_2^2 \dx \right)}^{\frac{1}{2}} ,
   \label{eq:HR11:1}
\end{equation}
where $\Abs{H(\bx)} = \sqrt{ {H(\bx)}^2 }$~\cite[Theorem~5.1.5]{HuaRus11} (the interested reader is referred to, for example,~\cite{CheSunXu07,ForPer01,Hua05a,HuaSun03,Mir12} for anisotropic error estimates for interpolation with linear and higher order finite elements).
Notice that $\Abs{H(\bx)} = H(\bx)$ in the current situation (symmetric and positive definite $H(\bx)$).

Further,
\begin{align*}
   \Norm{ {(F_K')}^T H(\bx) F_K' }_2 
   &=   \Norm{ {H(\bx)}^{\frac{1}{2}} F_K' }_2^2
   =    \Norm{ {H(\bx)}^{\frac{1}{2}} R_K^{-\frac{1}{2}} R_K^{\frac{1}{2}} F_K' }_2^2
   \\
   &\le \Norm{ {H(\bx)}^{\frac{1}{2}} R_K^{-\frac{1}{2}} }_2^2 
        \Norm{  R_K^{\frac{1}{2}} F_K' }_2^2
   \\
   &=   \Norm{ R_K^{-\frac{1}{2}} H(\bx) R_K^{-\frac{1}{2}} }_2
        \Norm{ {(F_K')}^T R_K F_K' }_2
   \\
   &=   \lambda_{\max}\bigl(  R_K^{-\frac{1}{2}} H(\bx) R_K^{-\frac{1}{2}} \bigr)
        \Norm{ {(F_K')}^T R_K F_K' }_2
   \\
   &=   \lambda_{\max} \bigl( R_K^{-1} H(\bx) \bigr)
        \Norm{ {(F_K')}^T R_K F_K' }_2
   \\
   &\le \Norm{ R_K^{-1} H(\bx) }_2
        \Norm{ {(F_K')}^T R_K F_K' }_2.
\end{align*}
Similarly,
\[
   \Norm{ {(F_K')}^{-1}}_2^2 
   = \Norm{ {(F_K')}^{-1} {(F_K')}^{-T} }_2
  \le \Norm{ {(F_K')}^{-T} R_K^{-1} {(F_K')}^{-1} }_2 \Norm{R_K}_2.
\]
Thus, \cref{eq:HR11:1} yields
\[
   \Abs{u - \Pi_h u}_{H^1(K)}^2 
      \leq C \Norm{ \FRFinv }_2 \Norm{R_K}_2 
         \int_K \Norm{ \FRF }_2^2 \Norm{ R_K^{-1} H(\bx)}_2^2 \dx .
\]
Using this, \cref{eq:ali:approx}, \cref{eq:ali:inverse}, \cref{eq:CRs}, 
the fact that the trace of any $d\times d$ symmetric and positive definite matrix $A$ is equivalent to its $l^2$ norm, viz., $\Norm{A}_2  \le \tr(A) \le d \Norm{A}_2$, and absorbing powers of $d$ into the generic constant $C$, we get
\begin{align*}
   \Abs{u - \Pi_h u}_{H^1(\Omega)}^2
   &= \sum_K \Abs{u - \Pi_h u}_{H^1(K)}^2 \\
   &\leq C \sum_K C_{ali}^{d-1} 
      \Abs{K}^{-\frac{2}{d}} {\det(R_K)}^{-\frac{1}{d}}\Norm{R_K}_2
      \times \Abs{K} C_{ali}^2 \Abs{K}^{\frac{4}{d}} 
      {\det(R_K)}^{\frac{2}{d}} C_{R+,K}^2 \\
   & = C C_{ali}^{d+1} 
      \sum_K \Abs{K}^{\frac{d+2}{d}} {\det(R_K)}^{\frac{1}{d}} \Norm{R_K}_2 C_{R+,K}^2\\
   &= C C_{ali}^{d+1}  
      \sum_K {\left( \Abs{K} {\det(R_K)}^{\frac{1}{d+2}} 
         \Norm{R_K}_2^{\frac{d}{d+2}} \right)}^\frac{d+2}{d} C_{R+,K}^2.
\end{align*}
Applying \cref{eq:equi:approx} to the above result and using \cref{CR+} gives
\begin{align*}
   \Abs{u - \Pi_h u}_{H^1(\Omega)}^2
   &\leq C C_{ali}^{d+1} 
      \sum_K {\left( \frac{C_{eq}}{N} \sum_\tK \abs{\tK}
         {\det(R_\tK)}^{\frac{1}{d+2}} \Norm{R_\tK}_2^{\frac{d}{d+2}}
         \right)}^{\frac{d+2}{d}} C_{R+,K}^2 \\
   &= C C_{ali}^{d+1}  C_{eq}^{\frac{d+2}{d}} N^{-\frac{2}{d}} \left (\frac{1}{N} \sum_{K\in \cT_h} C_{R+,K}^2\right )
      {\left(\sum_\tK \abs{\tK}
         {\det(R_\tK)}^{\frac{1}{d+2}} \Norm{R_\tK}_2^{\frac{d}{d+2}}
      \right)}^{\frac{d+2}{d}}\\
   &\le C C_{ali}^{d+1}  C_{eq}^{\frac{d+2}{d}} N^{-\frac{2}{d}} C_{R+}^2
      {\left(\sum_K \abs{K}
         {\det(R_K)}^{\frac{1}{d+2}} \Norm{R_K}_2^{\frac{d}{d+2}}
      \right)}^{\frac{d+2}{d}}\\
   &= C C_{ali}^{d+1} C_{eq}^{\frac{d+2}{d}} N^{-\frac{2}{d}} C_{R+}^2
      {\left( \sum_{K} \int_K {\det(R_K)}^{\frac{1}{d+2}} 
         \Norm{R_K}_2^{\frac{d}{d+2}} \dx \right)}^{\frac{d+2}{d}} .
\end{align*}
Further, assumption~\cref{eq:CRs} implies
\begin{equation}
   \det(R_K)
   \le \det\bigl(H(\bx)\bigr) \Norm{H^{-1}(\bx) R_K}_2^d 
   \leq C_{R-}^{-d} \det\bigl(H(\bx)\bigr)
   \label{eq:detR:detH} 
\end{equation}
and
\begin{equation}
   \Norm{R_K}_2 
   = \Norm{H(\bx) H^{-1}(\bx) R_K }_2
   \le \Norm{H(\bx)}_2 \Norm{H^{-1}(\bx) R_K}_2 
   \le C_{R-}^{-1} \Norm{H(\bx)}_2 . 
   \label{eq:detR:detH-1}
\end{equation}
Thus,
\begin{align*}
   \Abs{u - \Pi_h u}_{H^1(\Omega)}^2
   &\leq C C_{ali}^{d+1} C_{R+}^2 C_{eq}^{\frac{d+2}{d}} N^{-\frac{2}{d}} 
      {\left( C_{R-}^{\frac{-2d}{d+2}} 
         \int_\Omega {\left({\det(H(\bx))}^{\frac{1}{d}} \Norm{H(\bx)}_2
      \right)}^{\frac{d}{d+2}} \dx \right)}^{\frac{d+2}{d}} \\
   &= C C_{ali}^{d+1} C_{eq}^{\frac{d+2}{d}}
      {\left(\frac{C_{R+}}{C_{R-}}\right)}^2 N^{-\frac{2}{d}} 
      {\left( \int_\Omega \Norm{ {\det(H(\bx))}^{\frac{1}{d}}
         H(\bx)}_2^{\frac{d}{d+2}} \dx \right)}^{\frac{d+2}{d}},
\end{align*}
which, together with \cref{eq:cea-1}, gives \cref{eq:thm:H1}.
\qquad
\end{proof}

\subsection{Remarks}
\Cref{thm:H1} shows how a nonconvergent recovered Hessian works in mesh adaptation.
The error bound \cref{eq:thm:H1} is linearly proportional to the ratio $C_{R+}/C_{R-}$, which is a measure for the closeness of $R$ to $H$.
Thus, the finite element error changes gradually with the closeness of the recovered Hessian.
If $R$ is a good approximation to $H$ (but not necessarily convergent), then $C_{R+}/C_{R-} = \mathcal{O}(1)$ and the solution-dependent factor in the error bound is
\begin{equation}
   \Norm{{\det(H)}^{\frac{1}{d}} H}_{L^{\frac{d}{d+2}}(\Omega)}^{\frac{1}{2}}.
   \label{eq:factor:1}
\end{equation}
On the other hand, if $R$ is not a good approximation to $H$, solution-dependent factor in the error bound will be larger.
For example, consider $R = I$ (the identity matrix), which leads to the uniform mesh refinement.
In this case the condition \cref{eq:CRs} is satisfied with
\[
   C_{R+} = C_{R+,K}
      = \max_{\bx\in \Omega} \lambda_{\max} \bigl(H(\bx)\bigr)
   \quad \text{and} \quad
   C_{R-} 
      = \min_{\bx\in \Omega} \lambda_{\min} \bigl(H(\bx)\bigr),
\]
where $\lambda_{\max} \bigl(H(\bx)\bigr)$ and $\lambda_{\min} \bigl(H(\bx)\bigr)$ denote the maximum and minimum eigenvalues of $H(\bx)$, respectively.
Thus, for $R=I$ the solution-dependent factor in the bound \cref{eq:thm:H1} becomes 
\[
   \frac{\max_{\bx\in \Omega} \lambda_{\max} \bigl(H(\bx)\bigr)}
      { \min_{\bx\in \Omega} \lambda_{\min} \bigl(H(\bx)\bigr)}
   \Norm{{\det(H)}^{\frac{1}{d}} H}_{L^{\frac{d}{d+2}}(\Omega)}^{\frac{1}{2}},
\]
which is obviously larger than \cref{eq:factor:1}.

Next, we study the relation between \cref{eq:CRs} and \cref{eq:AgLiVa99:1}--\cref{eq:AgLiVa99:2}.
In practical computation, the Hessian is typically recovered at mesh nodes (see~\cref{sect:recovery:methods}) and a recovered Hessian can be considered on the whole domain as a piecewise linear matrix-valued function.
In this case, the average $R_K$ of $R$ over any given element $K$ can be expressed as a linear combination of the nodal values of $R$.
Applying the triangle inequality to \cref{eq:AgLiVa99:1,eq:AgLiVa99:2} we get
\[
   \Norm{R(\bx_i)-H(\bx)}_\infty 
      \le \left( \delta + \varepsilon \right) \lambda_{\min}(R_{\bx_i}),
      \quad \forall \bx \in \omega_i 
\]
and, since $R_K$ is a linear combination of $R(\bx_i)$, 
\[
   \Norm{R_K - H(\bx)}_\infty 
      \leq \left(\delta + \varepsilon \right) \lambda_{\min} ( R_K ) .
\]

Since $R_K - H(\bx)$ is symmetric, $\Norm{R_K - H(\bx)}_2 \le \Norm{R_K - H(\bx)}_\infty$.
Thus, conditions \cref{eq:AgLiVa99:1,eq:AgLiVa99:2} imply
\begin{equation}
   \Norm{ R_K - H(\bx) }_2 
   \leq \left( \delta + \varepsilon \right) \lambda_{\min} (R_K),
   \quad \forall \bx \in K, \quad \forall K \in \cT_h
   \label{eq:R:H:eps}
\end{equation}
and
\begin{equation}
   \Norm{ R_K^{-1} H(\bx) - I }_2 \leq \left( \delta + \varepsilon \right), 
      \quad \forall \bx \in K, \quad \forall K \in \cT_h
   \label{eq:R:H:I}
\end{equation}
which in turn implies \cref{eq:CRs} with $C_{R+,K} = 1+ \left( \delta + \varepsilon \right)$ and $C_{R-} =  1 - \left( \delta + \varepsilon \right)$, if $\left(\delta + \varepsilon\right) < 1$.
Condition \cref{eq:R:H:I} and therefore \cref{eq:AgLiVa99:2} require the eigenvalues of $\RKinvH$ to stay closely around one.
On the other hand, condition \cref{eq:CRs} only requires the eigenvalues of $\RKinvH$ to be bounded from above and below from zero, which is weaker than \cref{eq:AgLiVa99:2}.
If $R$ converges to $H(\bx)$, both \cref{eq:AgLiVa99:2,eq:CRs} can be satisfied.
However, if $R$ does not converge to $H(\bx)$, as is the case for most adaptive computation, the situation is different.
As we shall see in~\cref{sect:examples}, condition \cref{eq:CRs} is satisfied for all of the examples tested whereas condition \cref{eq:AgLiVa99:2} is not satisfied by either of the examples.

We would like to point out that it is unclear if the considered monitor function \cref{eq:M:H1} (and the corresponding bound \cref{eq:thm:H1}) is optimal, although it seems to be the best we can get.
For example, if we choose the monitor function to be
\begin{equation}
   M_K = {\det(R_K)}^{- \frac{1}{d+4}} R_K, \quad \forall K \in \cT_h
   \label{eq:M:L}
\end{equation}
which is optimal for the $L^2$ norm~\cite{Hua05a}, the error bound becomes
\begin{equation}
   \Abs{u - u_h}_{H^1(\Omega)}
   \leq C 
      \cdot C_{ali}^{\frac{d+1}{2}} C_{eq}^{\frac{d+4}{4d}}
      \cdot \frac{C_{R+}}{C_{R-}}
      \cdot N^{-\frac{1}{d}} 
      \Norm{{\det(H)}^{\frac{1}{d}}}_{L^{\frac{2d}{d+4}}(\Omega)}^{\frac{2}{d+4}}
         \Norm{ {\det(H)}^{\frac{1}{d+4}} H }_{L^1(\Omega)}^{\frac{1}{2}}.
   \label{thm:error:H1:2}
\end{equation}
This bound has a larger solution-dependent factor than \cref{eq:thm:H1} since Hölder's inequality yields
\[
\Norm{ {\det(H)}^{\frac{1}{d}} H}_{L^{\frac{d}{d+2}}(\Omega)}^{\frac{1}{2}}
\le \Norm{{\det(H)}^{\frac{1}{d}}}_{L^{\frac{2d}{d+4}}(\Omega)}^{\frac{2}{d+4}}
         \Norm{ {\det(H)}^{\frac{1}{d+4}}  H}_{L^1(\Omega)}^{\frac{1}{2}} .
\]
It is worth mentioning that when the metric tensor \cref{eq:M:L} is used, the $L^2$ norm of the piecewise linear interpolation error is bounded by
\begin{equation}
   \Norm{u - \Pi_h u}_{L^2(\Omega)} 
      \leq C 
         \cdot C_{ali} C_{eq}^{\frac{d+4}{2 d}}      
         \cdot \frac{C_{R+}}{C_{R-}}
         \cdot N^{-\frac{2}{d}}
      \Norm{ {\det(H)}^{\frac{1}{d}}}_{L^{\frac{2d}{d+4}}(\Omega)} ,
   \label{eq:error:L2}
\end{equation}
which is optimal in terms of convergence order and solution-dependent factor, e.g.,\ see~\cite{CheSunXu07,HuaSun03}.

Note that \cref{thm:H1} holds for $u \in H^2(\Omega)$ although the estimate \cref{eq:thm:H1} only requires
\begin{equation}
   \Norm{ {\det(H)}^{\frac{1}{d}} H }_{L^{\frac{d}{d+2}}(\Omega)} < \infty.
\label{reg-1}
\end{equation}
Since
\[
   \Norm{ {\det(H)}^{\frac{1}{d}} H }_{L^{\frac{d}{d+2}}(\Omega)}
      \le
      \Norm{ \frac{1}{d} \tr(H) \cdot H }_{L^{\frac{d}{d+2}}(\Omega)} ,
\]
\cref{reg-1} can be satisfied when $u \in W^{2, \frac{2d}{d+2}}(\Omega)$.
Thus, there is a gap between the sufficient requirement $u\in H^2(\Omega)$ and the necessary requirement $u \in W^{2, \frac{2d}{d+2}}(\Omega)$.
The stronger requirement $u \in H^2(\Omega)$ comes from the estimation of the interpolation error in~\cite[Theorem~5.1.5]{HuaRus11}.
It is unclear to the authors whether or not this requirement can be weakened.

It is pointed out that $u \in H^2(\Omega)$ may not hold when $\partial \Omega$ is not smooth.
For example, in 2D, if $\partial \Omega$ has a corner with an angle $\omega \in (0,2\pi)$, the solution of the BVP \cref{eq:bvp-2} with smooth $f$ and $g$ basically has the following form near the corner,
\[
   u(r, \theta) = r^{\frac{\pi}{\omega}} u_0(\theta) + u_1(r, \theta),
\]
where $(r,\theta)$ denote the polar coordinates and $u_0(\theta)$ and $u_1(r, \theta)$ are some smooth functions.
Then,
\[
   \Abs{u}_{H^2(\Omega)}^2  
      \sim \int_0^{b} {\left( r^{\frac{\pi}{\omega}-2} \right)}^2 r \,dr
      \sim \left. r^{\frac{2 \pi}{\omega}-2} \right\vert_{0}^{b}
\]
for some constant $b>0$.
This implies that $u \notin H^2(\Omega)$ if $\omega > \pi$.
On the other hand, $W^{2, \frac{2d}{d+2}}(\Omega) = W^{2, 1}(\Omega)$ for $d = 2$ and
\[
   \Abs{u}_{W^{2,1}(\Omega)}^2  
      \sim \int_0^{b} \left( r^{\frac{\pi}{\omega}-2} \right) r \,dr \\
      \sim \left. r^{\frac{\pi}{\omega}} \right |_{0}^{b},
\]
which indicates that $u \in W^{2,1}(\Omega)$ for all $\omega \in (0,2\pi)$.

\section{Convergence of the linear finite element approximation for a general Hessian}
\label{sect:general}

In this section we consider the general situation where $H(\bx)$ is symmetric but not necessarily positive definite.
In this case, it is unrealistic to require the recovered Hessian $R$ to be positive definite. 
Thus, we cannot use $R$ directly to define the metric tensor which is required to be positive definite. 
A commonly used strategy is to replace $R$ by $\Abs{R} = \sqrt{R^2}$ since $\Abs{R}$ retains the eigensystem of $R$.
However, $\Abs{R}$ can become singular locally.
To avoid this difficulty, we regularize $\Abs{R}$ with a regularization parameter $\alpha_h > 0$ (to be determined).

From \cref{eq:M:H1}, we define the regularized metric tensor as
\begin{equation}
   M_K = {\det(\alpha_h I + \Abs{R_K})}^{- \frac{1}{d+2}} 
      \Norm{\alpha_h I + \Abs{R_K} }_2^{\frac{2}{d+2}}
         \left(\alpha_h I + \Abs{R_K}\right),
   \quad \forall K \in \cT_h,
   \label{eq:M:H1:2}
\end{equation}
and obtain the following theorem with a proof similar to that of \cref{thm:H1}.

\begin{theorem}[General Hessian]
\label{thm:H1:general}
For a given positive parameter $\alpha_h > 0$, we assume that the recovered Hessian $R$ satisfies
\begin{align}
   & C_{R-,K} I \le  {\left( \alpha_h I + \Abs{R_K} \right)}^{-1} \left(\alpha_h I + \Abs{H(\bx)} \right), 
      \quad \forall \bx \in K,
      \quad \forall K \in \cT_h,
   \label{eq:CRminus:general}
   \\
   & {\left( \alpha_h I + \Abs{R_K} \right)}^{-1} \Abs{H(\bx)}  
      \leq C_{R+,K} I, 
      \quad \forall \bx \in K,
      \quad \forall K \in \cT_h,
   \label{eq:CRplus:general}
\end{align}
where $C_{R-,K}$ and $C_{R+,K}$ are element-wise constants satisfying  \cref{CR+}.
If the solution of the BVP \cref{eq:bvp-2} is in $H^2(\Omega)$, then for any quasi-$M$-uniform mesh associated with metric tensor \cref{eq:M:H1:2} and satisfying \cref{eq:equi:approx,eq:ali:approx} the linear finite element error for the BVP is bounded by
\begin{equation}
   \Abs{u - u_h}_{H^1(\Omega)} 
      \leq C
      \cdot C_{ali}^{\frac{d+1}{2}} C_{eq}^{\frac{d+2}{2d}}
      \cdot \frac{C_{R+}}{C_{R-}}
      \cdot N^{-\frac{1}{d}}
      \Norm{ {\det(\alpha_h I + \Abs{H})}^{\frac{1}{d}} 
         \left(\alpha_h I + \Abs{H}\right)
         }_{L^{\frac{d}{d+2}}(\Omega)}^{\frac{1}{2}}.
   \label{eq:thm:H1:general}
\end{equation}
\end{theorem}

From \cref{eq:CRplus:general,eq:CRminus:general,eq:thm:H1:general} we see that the greater $\alpha_h$ is, the easier the recovered Hessian satisfies \cref{eq:CRplus:general,eq:CRminus:general}; however, the error bound increases as well.
For example, consider the extreme case of $\alpha_h \to \infty$.
In this case, \cref{eq:CRplus:general,eq:CRminus:general} can be satisfied with $C_{R+} = C_{R-} = 1$ for any $R$.
At the same time, the metric tensor defined in \cref{eq:M:H1:2} has an asymptotic behavior $M_K \to \alpha_h^{\frac{4}{d+4}} I$ and the corresponding $M$-uniform mesh is a uniform mesh.
Obviously, the right-hand side of \cref{eq:thm:H1:general} is large for this case.
Another extreme case is $\alpha_h \to 0$ where \cref{eq:thm:H1:general} reduces to \cref{eq:thm:H1} if both $R$ and $H(\bx)$ are positive definite.

We now consider the choice of $\alpha_h$.
We define a parameter $\alpha$ through the implicit equation 
\begin{equation}
   \Norm{ \sqrt[d]{\det(\alpha I + \Abs{H})} 
      \cdot (\alpha I + \Abs{H}) }_{L^{\frac{d}{d+2}}(\Omega)}
   = 2 \Norm{ \sqrt[d]{\det(\Abs{H})} \cdot H}_{L^{\frac{d}{d+2}}(\Omega)}  .
   \label{eq:alpha-3}
\end{equation}
The left-hand-side term is an increasing function of $\alpha$.
Moreover, the term is equal to the half of the right-hand-side term when $\alpha = 0$ and tends to infinity as $\alpha \to \infty$. 
Thus, from the intermediate value theorem we know that \cref{eq:alpha-3} has a unique solution $\alpha > 0$ if $\Norm{ \sqrt[d]{\det(\Abs{H})} \cdot H}_{L^{\frac{d}{d+2}}(\Omega)} > 0$.
If we choose $\alpha_h = \alpha$, then the finite element error is bounded by
\begin{equation}
   \Abs{u - u_h}_{H^1(\Omega)} 
   \leq C 
      \cdot C_{ali}^{\frac{d+1}{2}} C_{eq}^{\frac{d+2}{2d}}
      \cdot \frac{C_{R+}}{C_{R-}}
      \cdot 2N^{-\frac{1}{d}}
      \Norm{ \sqrt[d]{\det(\Abs{H})} \cdot H
         }_{L^{\frac{d}{d+2}}(\Omega)}^{\frac{1}{2}},
   \label{eq:thm:H1:general-2}
\end{equation}
which is essentially the same as \cref{eq:thm:H1}.
Note that \cref{eq:alpha-3} is impractical since it requires the prior knowledge of $H(\bx)$.
In practice it can be replaced by
\begin{equation}
   \sum_K \Abs{K} {\det\left( \alpha_h I + \Abs{R_K}\right)}^{\frac{1}{d+2}} 
   	\Norm{\alpha_h I + \Abs{R_K}}_2^{\frac{d}{d+2}}
   = 2^{\frac{d}{d+2}} \sum_K \Abs{K}
      {\det(\Abs{R_K})}^{\frac{1}{d+2}}\Norm{R_K}_2^{\frac{d}{d+2}} .
   \label{alpha-2}
\end{equation}
This equation can be solved effectively using the bisection method.
Numerical results show that $\alpha_h$ is close to $\alpha$ (\cref{fig:tanh:alpha}).

\section{A selection of commonly used Hessian recovery methods}
\label{sect:recovery:methods}

In this section we give a brief description of four commonly used Hessian recovery algorithms for two-dimensional mesh adaptation.
The interested reader is referred to~\cite{Kam09,ValManDomDufGui07} for a more detailed description of these Hessian recovery techniques. 

Recall that the goal of the Hessian recovery in the current context is to find an approximation of the Hessian in mesh nodes using the linear finite element solution $u_h$.
The approximation of the Hessian on an element is calculated as the average of the nodal approximations of the Hessian at the vertices of the element.

\subsection*{QLS:\ quadratic least squares fitting to nodal values}
This method involves the fitting of a quadratic polynomial to nodal values of $u_h$ at a selection of neighboring nodes in the least square sense and subsequent differentiation.
The original purpose of the QLS was the gradient recovery (e.g.,\ see Zhang and Naga~\cite{ZhaNag05}).
However, it is easily adopted for the Hessian recovery by simply differentiating the fitting polynomial twice.
 
More specifically, for a given node (say $\bx_0$) at least five neighboring nodes are selected.
A quadratic polynomial (denoted by $p$) is found by least squares fitting to the values of $u_h$ at the selected nodes.
The linear system associated with the least squares problem usually has full rank and a unique solution.
If it does not, additional nodes from the neighborhood of $\bx_0$ are added to the selection until the system has full rank.
An approximation to the Hessian of the solution $u$ at $\bx_{0}$ is defined as the Hessian of $p$, viz.,
\[
   R^{QLS}(\bx_{0}) = H(p)(\bx_{0}) .
\]
\subsection*{DLF:\ double linear least squares fitting}
\label{ssect:DLF}

The DLF method computes the Hessian by using linear least squares fitting twice. 
First, the least squares fitting of the nodal values of $u_h$ in a neighbourhood of $\bx_0$ is employed to find a linear fitting polynomial $p$.
The recovered gradient of function $u$ at $\bx_0$ is defined as the gradient of $p$ at $\bx_{0}$, i.e., 
\[
   \nabla_h^{DLF} u(\bx_{0}) = \nabla p(\bx_{0}).
\]
Second-order derivatives are then obtained by subsequent application of this linear fitting to the calculated first order derivatives. 
Mixed derivatives are averaged in order to obtain a symmetric recovered Hessian.

\subsection*{LLS:\ linear least squares fitting to first-order derivatives}

This method is similar to DLF except that the first-order derivatives at nodes are calculated in a different way.
In this method, the first-order derivatives are first calculated at element centers and then at nodes by linear least squares fitting to their values at element centers.

\subsection*{WF:\ weak formulation}

This approach recovers the Hessian by means of a variational formulation~\cite{Dol98}.
More specifically, let $\phi_0$ be a canonical piecewise linear basis function at node $\bx_0$. 
Then the nodal approximation $u_{xx,h}$ to the second-order derivative $u_{xx} $ at $\bx_i$ is defined through 
\[
 u_{xx,h}(\bx_0) \int_\Omega \phi_0(\bx) \dx 
   = -\int_\Omega \frac{\partial u_h}{\partial x} 
      \frac{\partial \phi_0}{\partial x} \dx .
\]
The same approach is used to compute $u_{xy,h}$ and $u_{yy,h}$.
Since $\phi_0$ are piecewise linear and vanish outside the patch associated with $\bx_0$, the involved integrals can be computed efficiently with appropriate quadrature formulas over a single patch.

\section{Numerical examples}
\label{sect:examples}

In this section we present two numerical examples to verify the analysis given in the previous sections.
We use \bamg{}~\cite{bamg} to generate adaptive meshes as quasi-$M$-uniform meshes for the regularized metric tensor \cref{eq:M:H1:2}.
Special attention will be paid to mesh conditions \cref{eq:equi:approx,eq:ali:approx} and closeness conditions \cref{eq:AgLiVa99:2,eq:CRplus:general,eq:CRminus:general}.

For the recovery closeness condition \cref{eq:AgLiVa99:2} we compare the regularized recovered and exact Hessians, i.e.,\ we compute $\varepsilon$ for
\begin{equation}
   \Norm{(\alpha_h I + \Abs{R_K}) - (\alpha I + \abs{H_K}) }_\infty 
   \leq \varepsilon \lambda_{\min} (\alpha_h I + \Abs{R_K}),
   \qquad \forall K \in \cT_h,
   \label{eq:eps:regularized}
\end{equation}
where $H_K$ is an average of the exact Hessian on the element $K$ and $\alpha_h$ and $\alpha$ are the is the regularization parameters for the recovered and the exact Hessians, respectively.

\vspace{0.5ex}
\begin{example}[{\cite[Example~4.3]{Hua05a}}]
\label{ex:flower}
\normalfont{}
The first example is in the form of BVP~\cref{eq:bvp} with $f$ and $g$ chosen such that the exact solution is given by
\begin{align*}
   u(x,y) =& 
      \tanh \left[ 30 \left(   x^2       + y^2         - 0.125 \right) \right] \\
   &+ \tanh \left[ 30 \left( {(x-0.5)}^2 + {(y-0.5)}^2 - 0.125 \right) \right] \\
   &+ \tanh \left[ 30 \left( {(x-0.5)}^2 + {(y+0.5)}^2 - 0.125 \right) \right] \\
   &+ \tanh \left[ 30 \left( {(x+0.5)}^2 + {(y-0.5)}^2 - 0.125 \right) \right] \\
   &+ \tanh \left[ 30 \left( {(x+0.5)}^2 + {(y-0.5)}^2 - 0.125 \right) \right]
    .
\end{align*}

A typical plot of element-wise constants $C_{eq,K}$ and $C_{ali,K}$ in mesh quasi-$M$-uniformity conditions \cref{eq:equi:approx,eq:ali:approx} is shown in \cref{fig:flower:ceq,fig:flower:cali}, demonstrating that these conditions hold with relatively small $C_{eq}$ and $C_{ali}$.
 For the given mesh example we have $0.5 \le C_{eq,K} \le 1.5$ and $1 \le C_{ali,K} \le 1.3$, which gives $C_{eq} = 1.5$ and $C_{ali} = 1.3$.
In fact, we found that $C_{eq} \le 2.0$ and $C_{ali} \le 2.1$ for all computations in this paper, indicating that \bamg{} does a good job in generating quasi-$M$-uniform meshes for a given metric tensor.

\Cref{fig:flower:epsk,fig:flower:eps} show a typical distribution of element-wise values of $\varepsilon$ in \cref{eq:eps:regularized} and its values for a sequence of adaptive grids.
We observe that for all methods $\varepsilon$ is not small with respect to one, which violates the condition \cref{eq:AgLiVa99:2}.

Typical element-wise values $C_{R+,K}/C_{R-}$ and values of $C_{R+}/C_{R-}$ for a sequence of adaptive grids are shown in \cref{fig:flower:crk,fig:flower:CR}.
Notice that $C_{R+} / C_{R-}$ stays relatively small and bounded, thus satisfying the closeness conditions \cref{eq:CRplus:general,eq:CRminus:general}.

For this example, the finite element error $\Abs{u-u_h}_{H^1(\Omega)}$ is almost undistinguishable for meshes obtained by means of the exact and recovered Hessian (\cref{fig:flower:error}) and the approximated $\alpha_h$, computed through \cref{alpha-2}, is very close the value for the exact Hessian (\cref{fig:flower:alpha}).

\begin{figure}[p]\centering
   \begin{subfigure}[t]{0.5\linewidth}
   \centering
   \includegraphics[clip]{{{1211.2877f-flower-ce}}}
      \caption{element-wise $C_{eq,K}$ for \cref{eq:equi:approx}\label{fig:flower:ceq}}
   \end{subfigure}%
   %
   \begin{subfigure}[t]{0.5\linewidth}
   \centering
   \includegraphics[clip]{{{1211.2877f-flower-ca}}}
   \caption{element-wise $C_{ali,K}$ for \cref{eq:ali:approx}\label{fig:flower:cali}}
\end{subfigure}\\[1em]
   %
   \begin{subfigure}[t]{0.5\linewidth}
   \centering
   \includegraphics[clip]{{{1211.2877f-flower-epsk}}}
      \caption{element-wise $\varepsilon$
         for \cref{eq:eps:regularized}\label{fig:flower:epsk}}
   \end{subfigure}%
   %
   \begin{subfigure}[t]{0.5\linewidth}
   \centering
   \includegraphics[clip]{{{1211.2877f-flower-crk}}}
      \caption{element-wise $C_{R+,K}/C_{R-}$\label{fig:flower:crk}}
   \end{subfigure}\\[1em]
   %
   \begin{subfigure}[t]{0.5\textwidth}
   \centering
   \includegraphics[clip]{{{1211.2877f-flower-eps}}}
      \caption{$\varepsilon$ for 
         \cref{eq:eps:regularized}\label{fig:flower:eps}}
   \end{subfigure}%
   %
   \begin{subfigure}[t]{0.5\textwidth}
   \centering
   \includegraphics[clip]{{{1211.2877f-flower-cr}}}
      \caption{$C_{R+}/C_{R-}$ for 
         \cref{eq:CRplus:general,eq:CRminus:general}\label{fig:flower:CR}}
   \end{subfigure}\\[1em]
   %
   \begin{subfigure}[t]{0.5\textwidth}
   \centering
   \includegraphics[clip]{{{1211.2877f-flower-fe-error}}}
      \caption{finite element error $\Abs{u-u_h}_{H^1(\Omega)}$\label{fig:flower:error}}
   \end{subfigure}%
   %
   \begin{subfigure}[t]{0.5\textwidth}
   \centering
   \includegraphics[clip]{{{1211.2877f-flower-alpha}}}
      \caption{comparison of $\alpha_h$ and $\alpha$\label{fig:flower:alpha}}
   \end{subfigure}%
   \caption{Numerical results for \cref{ex:flower}\label{fig:flower}}
\end{figure}
\end{example}

\vspace{0.5ex}
\begin{example}[Strong anisotropy]
\label{ex:tanh}
\normalfont{}
The second example is in the form of BVP~\cref{eq:bvp} with $f$ and $g$ chosen such that the exact solution is given by
\[
   u(x,y) = \tanh ( 60y ) - \tanh \bigl( 60 (x - y) - 30 \bigr).
\]
This solution exhibits a very strong anisotropic behavior and describes the interaction between a boundary layer along the $x$-axis and a steep shock wave along the line $y = x - 1/2$.

\Cref{fig:tanh:epsk,fig:tanh:epsilon} show that $\varepsilon \approx 60$
and therefore not small with respect to one, violating the condition \cref{eq:AgLiVa99:2} for all meshes in the considered range of $N$ for all four recovery techniques.

On the other hand, \cref{fig:tanh:CR} shows that the ratio $C_{R+} / C_{R-}$ is large ($\approx 10^2$) but, nevertheless, it seems to stay bounded with increasing $N$, confirming that \cref{eq:CRplus:general,eq:CRminus:general} are satisfied by the recovered Hessian. 
The fact that the ratio $C_{R+} / C_{R-}$ has different values in this and the previous examples indicates that the accuracy or closeness of the four Hessian recovery techniques depends on the behavior and especially the anisotropy of the solution.
Fortunately, as shown by \cref{thm:H1,thm:H1:general}, the finite element error is insensitive to the closeness of the recovered Hessian.
The finite element solution error is shown in \cref{fig:tanh:error} as a function of $N$.

Finally, \cref{fig:tanh:alpha} shows that $\alpha_h$, computed through \cref{alpha-2}, is close to the exact value $\alpha$ defined in \cref{eq:alpha-3}.

\begin{figure}[p]\centering
   %
   \begin{subfigure}[t]{0.5\linewidth}
   \centering
   \includegraphics[clip]{{{1211.2877f-tanh-ce}}}
      \caption{element-wise $C_{eq,K}$ for \cref{eq:equi:approx}\label{fig:tanh:ceq}}
   \end{subfigure}%
   %
   \begin{subfigure}[t]{0.5\linewidth}
   \centering
   \includegraphics[clip]{{{1211.2877f-tanh-ca}}}
   \caption{element-wise $C_{ali,K}$ for \cref{eq:ali:approx}\label{fig:tanh:cali}}
   \end{subfigure}\\[1em]
   %
   \begin{subfigure}[t]{0.5\linewidth}
   \centering
   \includegraphics[clip]{{{1211.2877f-tanh-epsk}}}
      \caption{element-wise $\varepsilon$
          for \cref{eq:eps:regularized}\label{fig:tanh:epsk}}
   \end{subfigure}%
   %
   \begin{subfigure}[t]{0.5\linewidth}
   \centering
   \includegraphics[clip]{{{1211.2877f-tanh-crk}}}
   \caption{element-wise $C_{R+,K}/C_{R-}$}
   \end{subfigure}\\[1em]
   \begin{subfigure}[t]{0.5\textwidth}
   \centering
   \includegraphics[clip]{{{1211.2877f-tanh-eps}}}
      \caption{$\varepsilon$
          for \cref{eq:eps:regularized}\label{fig:tanh:epsilon}}
   \end{subfigure}%
   %
   \begin{subfigure}[t]{0.5\textwidth}
   \centering
   \includegraphics[clip]{{{1211.2877f-tanh-cr}}}
      \caption{$C_{R+}/C_{R-}$ for 
         \cref{eq:CRplus:general,eq:CRminus:general}\label{fig:tanh:CR}}
   \end{subfigure}\\[1em]
   %
   \begin{subfigure}[t]{0.5\textwidth}
   \centering
   \includegraphics[clip]{{{1211.2877f-tanh-fe-error}}}
      \caption{finite element error $\Abs{u-u_h}_{H^1(\Omega)}$\label{fig:tanh:error}}
   \end{subfigure}%
   %
   \begin{subfigure}[t]{0.5\textwidth}
   \centering
   \includegraphics[clip]{{{1211.2877f-tanh-alpha}}}
      \caption{comparison of $\alpha_h$ and $\alpha$\label{fig:tanh:alpha}}
   \end{subfigure}%
   \caption{Numerical results for \cref{ex:tanh}\label{fig:tanh}}
\end{figure}
\end{example}

\section{Conclusion and further comments}
\label{sect:conclusion}

In the previous sections we have investigated how a nonconvergent recovered Hessian works in mesh adaptation.
Our main results are \cref{thm:H1,thm:H1:general} where an error bound for the linear finite element solution of BVP \cref{eq:bvp-2} is given for quasi-$M$-uniform meshes corresponding to a metric depending on a recovered Hessian.
As conventional error estimates for the $H^1$ semi-norm of the error in linear finite element approximations, our error bound is of first order in terms of the average element diameter, $N^{-\frac{1}{d}}$, where $N$ is the number of elements and $d$ is the dimension of the physical domain.
This error bound is valid under the closeness condition \cref{eq:CRs} (or \cref{eq:CRplus:general,eq:CRminus:general}), which is weaker than \cref{eq:AgLiVa99:2} used by Agouzal et al.~\cite{AgoLipVas99} and Vassilevski and Lipnikov~\cite{VasLip99}.
Numerical results in~\cref{sect:examples} show that the new closeness condition is satisfied by the recovered Hessian obtained with commonly used Hessian recovery algorithms.
The error bound also shows that the finite element error changes gradually with the closeness of the recovered Hessian to the exact one.
These results provide an explanation on how a nonconvergent recovered Hessian works in mesh adaptation.

In this work the closeness conditions \cref{eq:CRplus:general,eq:CRminus:general} have been verified only numerically.
Developing a theoretical proof of the condition for some Hessian recovery techniques is an interesting topic for further investigations.

\section*{Acknowledgment}
The authors are grateful to the anonymous referees for their comments and suggestions for improving the quality of this paper, particularly for the helpful comments on improving the proof of \cref{thm:H1}.


\end{document}